\documentclass[11pt]{amsart}
\usepackage{amsmath, amsthm, mathabx,amscd, amsfonts, amssymb, hyperref, mathrsfs, textcomp, epsfig, a4wide, graphicx, IEEEtrantools, tikz, verbatim}
\usepackage[mathscr]{euscript}

\usetikzlibrary{matrix}

\newtheorem{thm}{Theorem}
\newtheorem{prop}{Proposition}
\newtheorem{lemma}{Lemma}

\newcommand{\spin}{\mathfrak{s}}
\theoremstyle{definition}

\newtheorem{conj}{Conjecture}

\theoremstyle{remark}
\newtheorem{remark}{Remark}
\newtheorem{example}{Example}

     \RequirePackage{rotating}                   
    \def\HSt{%
       \setbox0=\hbox{$\widehat{\mathit{HS}}$}
       \setbox1=\hbox{$\mathit{HS}$}
       \dimen0=1.1\ht0
       \advance\dimen0 by 1.17\ht1
       \smash{\mskip2mu\raise\dimen0\rlap{%
          \begin{turn}{180}
              {$\widehat{\phantom{\mathit{HS}}}$}
           \end{turn}} \mskip-2mu    
                \mathit{HS}
    }{\vphantom{\widehat{\mathit{HS}}}}{}}

     \RequirePackage{rotating}                   
    \def\HMt{%
       \setbox0=\hbox{$\widehat{\mathit{HM}}$}
       \setbox1=\hbox{$\mathit{HM}$}
       \dimen0=1.1\ht0
       \advance\dimen0 by 1.17\ht1
       \smash{\mskip2mu\raise\dimen0\rlap{%
          \begin{turn}{180}
              {$\widehat{\phantom{\mathit{HM}}}$}
           \end{turn}} \mskip-2mu    
                \mathit{HM}
    }{\vphantom{\widehat{\mathit{HM}}}}{}}

         \RequirePackage{rotating}                   
    \def\HMIt{%
       \setbox0=\hbox{$\widehat{\mathit{HMI}}$}
       \setbox1=\hbox{$\mathit{HMI}$}
       \dimen0=1.1\ht0
       \advance\dimen0 by 1.17\ht1
       \smash{\mskip2mu\raise\dimen0\rlap{%
          \begin{turn}{180}
              {$\widehat{\phantom{\mathit{HMI}}}$}
           \end{turn}} \mskip-2mu    
                \mathit{HMI}
    }{\vphantom{\widehat{\mathit{HMI}}}}{}}

\newcommand{\HStil}{\widetilde{\mathit{HS}}}
    \newcommand{\HMIb}{\overline{\mathit{HMI}}}
\newcommand{\HMIf}{\widehat{\mathit{HMI}}}
\newcommand{\HMItil}{\widetilde{\mathit{HMI}}}
\newcommand{\Ctil}{\widetilde{C}}

\newcommand{\HMtil}{\widetilde{\mathit{HM}}}
\newcommand{\Pin}{\mathrm{Pin}(2)}

\newcommand{\q}{\mathfrak{q}}
\newcommand{\p}{\mathfrak{p}}
\newcommand{\ztwo}{\mathbb{F}}

\newcommand{\BN}{\widetilde{\mathrm{BN}}}

\newcommand{\Br}{\ztwo[Q]/Q^2}
\newcommand{\vm}{\mathbf{v}_-}
\newcommand{\vp}{\mathbf{v}_+}
\newcommand{\Skein}[1]{\raisebox{-3ex}{\includegraphics[height=7ex]{#1}}}

\begin{document}
\title[Khovanov homology and involutions]{Khovanov homology in characteristic two and involutive monopole Floer homology}

\author{Francesco Lin}
\address{Department of Mathematics, Princeton University and School of Mathematics, Institute for Advanced Study} 
\email{fl4@math.princeton.edu}

\begin{abstract}
We study the conjugation involution in Seiberg-Witten theory in the context of the Ozsv\'ath-Szab\'o and Bloom's spectral sequence for the branched double cover of a link $L$ in $S^3$. We prove that there exists a spectral sequence of $\Br$-modules (where $Q$ has degree $-1$) which converges to $\HMItil_*(\Sigma(L))$, an involutive version of the monopole Floer homology of the branched double cover, and whose $E^2$-page is a version of Bar Natan's characteristic two Khovanov homology of the mirror of $L$. We conjecture that an analogous result holds in the setting of $\Pin$-monopole Floer homology.
\end{abstract}

\maketitle

The interaction between quantum and Floer theoretic invariants in low-dimensional topology has spurred a lot of research activity in recent years. Perhaps the most well-studied from this perspective are the various spectral sequences starting from Khovanov homology (\cite{Kho},\cite{Bar}) and converging to (versions of) Heegaard Floer homology (\cite{OSBr}), monopole Floer homology (\cite{Blo}), and singular instanton homology (\cite{KMK}). These have been recently organized in a broad conceptual picture (\cite{BHL}).
\par
Another fruitful trend in the past few years has been the study of $\Pin$-symmetry in Seiberg-Witten Floer homology, where $\Pin$ is $S^1\cup j\cdot S^1\subset \mathbb{H}$. This was employed by Manolescu (\cite{Man}) to disprove the longstanding Triangulation conjecture. Morse-theoretic (\cite{KM}) counterparts of his invariants were constructed by the author in \cite{Lin}, and the extra symmetry has also been employed with a different flavor to define an involutive version of Heegaard Floer homology \cite{HM}. Intuitively, the latter corresponds to the subgroup $\mathbb{Z}_4=\langle j\rangle\subset \Pin$.
\\
\par
The goal of this paper is to find a common ground for the two aforementioned topics. Before stating the main result, we briefly introduce its protagonists. In what follows, $\ztwo$ is the field with two elements.
\begin{itemize}
\item For each link $L$ in $S^3$ there is a combinatorially defined invariant $\BN^2_{*,*}({L})$ introduced by Bar Natan \cite{BN2}, which is a bigraded module over $\ztwo[u]/u^2$, where  $u$ acts with bidegree $(-2,0)$;
\item For each compact oriented three-manifold $Y$ we will define in the present paper a Floer theoretic invariant $\HMItil_*(Y)$, called \textit{involutive monopole Floer homology}. This is a relatively graded module over $\ztwo[Q]/Q^2$, where $Q$ has degree $-1$ (to be precise, we also need to fix a basepoint in the three-manifold to define this invariant).
\end{itemize}
For a link $L$ in $S^3$ denote its mirror by $\bar{L}$, and its branched double cover by $\Sigma(L)$.
\begin{thm}
For any link $L\subset S^3$, there exists a spectral sequence whose $E^2$-page is isomorphic to $\BN^2_{*,*}(\bar{L})$ and which converges to $\HMtil_*(\Sigma(L))$. After setting $u$ to be $Q$, this is a spectral sequence of $\Br$-modules.
\end{thm}

We discuss more in detail the objects involved. First, for each link $L$ there is an unreduced invariant $\mathrm{BN}^2_{*,*}(L)$ which is a bigraded module over $\ztwo[u]/u^2$. This is obtained from Bar-Natan's variant of Khovanov homology $\mathrm{BN}_{*,*}(L)$ (\cite{BN2}, where we use the notation of \cite{Turn}), which is defined \textit{only over fields of characteristic two}, by setting $u^2$ to be zero. Given a diagram $D$ for $L$ , $\mathrm{BN}^2_{*,*}(L)$ arises as the homology of a chain complex whose underlying vector space is $
CKh(D)\otimes \ztwo[u]/u^2$ (where $CKh(D)$ is the Khovanov chain complex, see for example \cite{Bar}), and whose differential is determined by the following modification of the Khovanov's \textsc{tqft}:
\begin{equation}\label{BN}
m:V\otimes V\otimes\ztwo[u]/u^2\rightarrow V\otimes\ztwo[u]/u^2\quad
 \left\{ \begin{array}{rl}
      \vp\otimes\vp&\mapsto \vp\\
\vp\otimes \vm&\mapsto \vm\\
      \vm\otimes \vp&\mapsto \vm\\ 
     \vm\otimes \vm&\mapsto u\vm
    \end{array}\right.
\end{equation}
\begin{equation}
    \Delta:V\otimes\ztwo[u]/u^2\rightarrow V\otimes V\otimes\ztwo[u]/u^2\quad
   \left\{ \begin{array}{rl}
      \vp&\mapsto \vp\otimes \vm+\vm\otimes \vp+u\vp\otimes\vp\\
      \vm&\mapsto \vm\otimes \vm
         \end{array} \right.
\end{equation}
Here as usual $V$ is a vector space generated by $\{\vp,\vm\}$, which have degree respectively $\pm1$. The graded Euler characteristic of $\mathrm{BN}^2_{*,*}$ is
\begin{equation*}
(1+q^{-2})\cdot V(L)
\end{equation*}
where $V(L)$ is the unnormalized Jones polynomial of $L$, and furthermore this invariant fits in an exact triangle
\begin{center}
\begin{tikzpicture}
\matrix (m) [matrix of math nodes,row sep=1em,column sep=1em,minimum width=2em]
  {
  \mathrm{BN}^2_{*,*}(L) && \mathrm{BN}^2_{*,*}(L)[-1]\\
  &\mathrm{Kh}(L)_{*,*}&\\};
  \path[-stealth]
  (m-1-1) edge node [above]{} (m-1-3)
  (m-2-2) edge node [left]{} (m-1-1)
  (m-1-3) edge node [right]{} (m-2-2)  
  ;
\end{tikzpicture}
\end{center}
Some interesting results regarding Bar Natan's homology are discussed in \cite{Turn}: for example, it can be used to adapt Lee's spectral sequence (\cite{Lee}) in characteristic two.
The statement of the main theorem involves the reduced version of the homology $\BN^2_{*,*}(L)$: for a given basepoint in $L$, this is the homology of the complex obtained after quotienting by the subcomplex of elements that have $\vm$ in the cycle containing the basepoint. This construction is in fact independent of the choice of basepoint (cfr. the result for Khovanov homology in \cite{ORS}).
\\
\par
On the Floer theory side, we will define in the present paper for a compact connected and oriented three-manifold $Y$ (together with a choice of basepoint $p$) the invariant $\HMItil_*(Y,p)$, which is a relatively graded module over $\ztwo[Q]/Q^2$, $Q$ having degree $-1$. This is the analogue in monopole Floer homology of the hat version of involutive Heegaard Floer homology \cite{HM}. There is a natural conjugation involution $\jmath$ on the set of spin$^c$ structure inducing canonical isomorphisms
\begin{equation*}
\jmath_*:\HMtil_*(Y,\spin)\rightarrow \HMtil_*(Y,\bar{\spin}),
\end{equation*}
where $\HMtil$ is the \textit{tilde} monopole Floer homology group introduced in \cite{Blo} in analogy with the hat version of Heegaard Floer homology (\cite{OS}). The most interesting ones for us are the spin$^c$ structures for which $\bar{\spin}=\spin$, which we call \textit{self-conjugate}. The main idea is to exploit this involution at the chain level in order to recover some version of $\mathbb{Z}_2$-equivariant homology (see for example \cite{SS}). As in \cite{HM}, we only consider a truncated version of the resolution of the Floer chain complex, and we expect the full resolution complex to recover some version of the $\Pin$-monopole Floer homology defined in \cite{Lin}. Among the basic properties, we have that it fits in an exact triangle
\begin{center}\label{triangle}
\begin{tikzpicture}
\matrix (m) [matrix of math nodes,row sep=1em,column sep=1em,minimum width=2em]
  {
  \HMtil_{*}(Y,p) && \HMtil_{*}(Y,p)[-1]\\
  &\HMItil_{*}(Y,p) &\\};
  \path[-stealth]
  (m-1-1) edge node [above]{$1+\jmath_*$} (m-1-3)
  (m-2-2) edge node [left]{} (m-1-1)
  (m-1-3) edge node [right]{} (m-2-2)  
  ;
\end{tikzpicture}
\end{center}
Furthermore, for $S^3$ we have $\HMItil(S^3)=\Br$, and cobordisms $W$ equipped with a properly embedded arc $\gamma$ from the incoming component to the outgoing one induce maps in a functorial fashion. For completeness, we will also introduce the \textit{to}, \textit{from} and \textit{bar} versions of the invariants, which correspond respectively to the $+,-$ and $\infty$ flavors on the Heegaard Floer side. 
\begin{remark}We will generally drop the basepoint from the notation when it is non strictly necessary.
\end{remark}

\begin{example}
We discuss the simplest (and easily generalizable) example. For $K$ the trefoil knot, one can compute that $\BN^2_{*,*}(K)$ is
\begin{center}
\begin{tikzpicture}
\matrix (m) [matrix of math nodes,row sep=0.5em,column sep=0.5em,minimum width=2em]
  {{}&-3&-2&-1&0\\
  1&&&&\ztwo\\
  -1&&&&\ztwo\\
  -3&&&&\\
  -5&&\ztwo&&\\
  -7&&&&\\
  -9&\ztwo&&&\\};
  \path[-stealth]
  (m-2-5.east) edge[bend left] (m-3-5.north east);
  \draw(m-1-1.south west)--(m-1-5.south east);
  \draw(m-1-1.north east)--(m-7-1.south east);
  \end{tikzpicture}
\end{center}
where the arrow indicates the action of $u$. Now, as $K$ is alternating, $Y=\Sigma(K)$ is an $L$-space (see \cite{OSBr} for the proof in the Heegaard Floer setting). As $K$ has determinant $3$, it has one self-conjugate spin$^c$ structure and a pair of non-isomorphic conjugate ones. It will be shown that the former contributes $\Br$ to $\HMItil_*(\Sigma(L))$ (as it is the case for the only self-conjugate spin$^c$ structure on $S^3$), while the latter contributes $\ztwo\oplus\ztwo[-1]$. In particular, the spectral sequence collapses at the $E^2$-page.
\end{example}
\vspace{0.5cm}

The construction of the surgery spectral sequence mentioned in the first point closely follows the work in \cite{Blo}. The key difference arises when one tries to identify the $E^1$ page of the spectral sequence associated to the branched double cover. This is because for involutive monopole Floer homology the map
\begin{equation*}
\HMItil_*(S^2\times S^2\setminus (D^4\amalg D^4)): \HMItil_*(S^3)\rightarrow \HMItil_*(S^3)
\end{equation*}
is multiplication by $Q$ on $\HMItil_*(S^3)=\Br$ (cfr. Proposition $3$ in \cite{Lin2} for the analogous statement in the $\Pin$-setting). This is not the case in $\HMtil$, as there the induced map is zero because $b_2^+$ of the cobordism is strictly positive (see for example Chapter $3$ of \cite{KM}). As $S^2\times S^2$ is the branched double cover of $S^4$ over a standardly embedded torus, this implies that when describing the $E^1$-page of the spectral sequence there will be some additional terms containing $Q$. These correspond, after suitable identifications, to the terms involving $u$ in the maps (\ref{BN}).
\par
Related to this, we have an isomophism
\begin{equation}\label{ident}
\HMItil_*(S^2\times S^1)= \Lambda^* (\ztwo\langle\gamma\rangle)\otimes \Br
\end{equation}
for $\gamma$ a generator of $H_1(S^2\times S^1)$ (in degree $-1$) (see Chapter $4$ of \cite{Lin} for the analogue in the $\Pin$-setting), but such an isomorphism is not canonical. Indeed $S^2\times S^1$ has two spin structures (both inducing the unique self-conjugate spin$^c$ structure), and there is an orientation preserving self-diffeomorphism $\varphi$ of $S^2\times S^1$ exchanging them (\cite{GS}). More explicitly, if $\gamma:S^1\rightarrow SO(3)$ is a generator of $\pi_1(SO(3))= \mathbb{Z}_2$, we can take $\varphi$ to be
\begin{equation*}
\varphi: S^2\times S^1\rightarrow S^2\times S^1\qquad (x,\vartheta)\mapsto (\gamma(\vartheta)\cdot x, \vartheta).
\end{equation*}
Then $\varphi$ induces the non trivial automorphism (as a $\Br$-module) of $\HMItil_*(S^2\times S^1)$ given by
\begin{equation*}
\gamma\mapsto \gamma+Q.
\end{equation*}
Hence to fix an identification as in (\ref{ident}), and more generally for a connected sum of $S^2\times S^1$s, we need to fix a spin structure. This is unlike the case for the identification
\begin{equation*}
\HMtil_*(S^2\times S^1)= \Lambda^* (\ztwo\langle\gamma\rangle),
\end{equation*}
as this is unique for degree reasons.
\\
\par
The reader may wonder why we are working in the involutive setting instead of the $\Pin$-one. The main complication in the latter is that while one can estabilish a surgery exact triangle (see \cite{Lin2}), one of the maps appearing in the statement is generally not the one induced by the corresponding $2$-handle attachment cobordism. This is related to the fact that the blow-up formula in that setting shows that the composition of two consecutive handle attachment maps in the triangle might not be zero. Nevertheless, its image is contained in the image of $Q^2$, and as $Q^2$ is zero in the involutive setting one might expect this setting to be more tractable. Most of the technical work in the present paper relies on the fact that the composition is in fact explicitly chain homotopic to zero. From a more analytical perspective, the difference is related to the fact that a certain space of self-adjoint operators that naturally arises in the construction is simply-connected, but has non-trivial higher homotopy groups. In any case, we are lead to state the following conjecture.
\begin{conj}
Given a link $L\subset S^3$, there exists a spectral sequence whose $E^2$-page is $\BN^3_{*,*}(\bar{L})$ which converges to $\HStil_*(\Sigma(L))$. After setting $u$ to be $Q$, this is a spectral sequence of $\ztwo[Q]/Q^3$-modules.
\end{conj}
Here $\BN^3(\bar{L})$ is the reduced version of Bar Natan's homology in which we set $u^3$ to be zero, while $\HStil$ is the defined to be the mapping cone of a the chain map inducing the action of $V$ in $\Pin$-monopole Floer homology. One can study the analogous spectral sequence for the branched double cover in the $\Pin$-setting: while the arguments in the present paper allow one to identify the $E^2$ page in exactly the same way, the non trivial part is to show that it converges to the Floer homology of the branched double cover. The last step is indeed where the surgery exact triangle is needed.
\\
\par
\textit{Organization of the paper. }In Section \ref{involutive} we define involutive monopole Floer homology and discuss its main properties. In Section \ref{surgery}, we prove the link surgery spectral sequence in involutive monopole Floer homology for a link in a three-manifold. Finally, in Section \ref{barnatan} we study the spectral sequence associated to the branched double cover, and prove the main theorem.
\\
\par
\textit{Acknowledgements.} The author would like to thank Paolo Aceto, Ciprian Manolescu, Peter Ozsv\'ath and Zolt\'an Szab\'o for the helpful discussions.

\vspace{0.5cm}
\section{Involutive monopole Floer homology}\label{involutive}

The goal of this section is to introduce a package of invariants of compact, connected oriented three-manifolds called \textit{involutive monopole Floer homology}. These invariants should be considered as the analogue in our theory of involutive Heegaard Floer homology \cite{HM}. We start by discussing some formal properties, we then quickly review some constructions in monopole Floer homology, and finally we proceed with definitions and proofs. We assume the reader has some familiarity with monopole Floer homology (\cite{KM}); more specifically, the content of the notes \cite{Lin4} will be more than sufficient.
\\
\par
\textbf{Formal properties.}
For a given three-manifold $Y$, there exists a natural conjugation action $\jmath$ on the set of spin$^c$ structures $\mathrm{Spin}^c(Y)$. We call the fixed points \textit{self-conjugate spin$^c$ structures}. There exists a natural map
\begin{equation*}
\{\text{spin structures on $Y$}\}\rightarrow \{\text{self-conjugate spin$^c$ structures on $Y$}\}
\end{equation*}
which is surjective and such that each fiber has cardinality $2^{b_1(Y)}$ (cfr. Chapter $4$ of \cite{Lin}).
\par
To each three-manifold $Y$ equipped with $[\spin]\in \mathrm{Spin}^c(Y)/\jmath$, we associate relatively $\mathbb{Z}/d([\spin])\mathbb{Z}$-graded modules
\begin{equation*}
 \HMIt_*(Y,[\spin])\qquad \HMIf_*(Y,[\spin])\qquad \HMIb_*(Y,[\spin])
\end{equation*}
over $\ztwo[U,Q]/Q^2$, where $Q$ has degree $-1$ and $U$ has degree $-2$. Here $d([\spin])$ is the least common factor of all $\langle c_1(\spin),x\rangle$ for $x\in H_1(Y,\mathbb{Z})$. Furthermore, when $c_1$ is torsion (e.g. when the structure is self-conjugate) the relative $\mathbb{Z}$-grading can be lifted to an absolute $\mathbb{Q}$-grading. These are called \textit{involutive monopole Floer homology groups}, and fit in an exact sequence of $\ztwo[U,Q]/Q^2$-modules
\begin{center}
\begin{tikzpicture}
\matrix (m) [matrix of math nodes,row sep=2em,column sep=1em,minimum width=2em]
  {
  \HMIt_{*}(Y,[\spin]) && \HMIf_{*}(Y,[\spin])\\
  &\HMIb_{*}(Y,[\spin]) &\\};
  \path[-stealth]
  (m-1-1) edge node [above]{$j_*$} (m-1-3)
  (m-2-2) edge node [left]{$i_*$} (m-1-1)
  (m-1-3) edge node [right]{$p_*$} (m-2-2)  
  ;
\end{tikzpicture}
\end{center}
where $i_*$ and $p_*$ have degree zero, while $j_*$ has degree one. When $[\spin]$ is not self-conjugate, there are natural identifications as $\ztwo[U,Q]/Q^2$-modules
\begin{equation}\label{nonspin}
\HMIt_{*}(Y,[\spin]) \equiv\HMIt_{*}(Y,\spin)\oplus \HMIt_{*}(Y,\spin)[-1]\equiv\HMIt_{*}(Y,\bar{\spin})\oplus \HMIt_{*}(Y,\bar{\spin})[-1]
\end{equation}
where the action of $Q$ is trivial. For a self conjugate spin$^c$ structure $\spin$, the exact triangle
\begin{center}
\begin{tikzpicture}
\matrix (m) [matrix of math nodes,row sep=1em,column sep=1em,minimum width=2em]
  {
  \HMt_{*}(Y,\spin) && \HMt_{*}(Y,\spin)[-1]\\
  &\HMIt_{*}(Y,\spin) &\\};
  \path[-stealth]
  (m-1-1) edge node [above]{$1+\jmath_*$} (m-1-3)
  (m-2-2) edge node [left]{$$} (m-1-1)
  (m-1-3) edge node [right]{$$} (m-2-2)  
  ;
\end{tikzpicture}
\end{center}
holds, where
\begin{equation*}
\jmath_*:\HMt_*(Y,\spin)\rightarrow \HMt_*(Y,{\spin}),
\end{equation*}
is the conjugation isomorphism. We can also define the total Floer groups
\begin{equation*}
\HMIt_*(Y)=\bigoplus_{[\spin]\in \mathrm{Spin}^c(Y)/\mathbb{\jmath}}\HMIt_*(Y,[\spin]).
\end{equation*}
The same statements hold for the other flavors. As a basic example, we have
\begin{align*}\label{S3}
\HMIb_*(S^3)&=\ztwo[Q,U^{-1},U]/Q^2\\
\HMIf_*(S^3)&=\ztwo[Q,U]/Q^2\\
\HMIt_*(S^3)&=\left(\ztwo[Q,U^{-1},U]/Q^2\right)/\left(\ztwo[Q,U]/Q^2\right),
\end{align*}
where the gradings are shifted so that the bottom element of the last group has degree zero. As in the usual case, one can define completed version of the invariants, which are denoted with a bullet. The invariant is functorial, meaning that a cobordism $W$ from $Y_0$ to $Y_1$ induces a map of $\ztwo[U,Q]/Q^2$-modules
\begin{equation*}
\HMIt_{\bullet}(W):\HMIt_{\bullet}(Y_0)\rightarrow \HMIt_{\bullet}(Y_1).
\end{equation*}
This map respects the decomposition of spin$^c$ structures up to conjugation on $W$, and furthermore if $W_0$ and $W_1$ are composable, i.e. the outgoing boundary of $W_0$ coincides with the incoming boundary of $W_1$, then
\begin{equation*}
\HMIt_{\bullet}(W_1\circ W_0)=\HMIt_{\bullet}(W_1)\circ\HMIt_{\bullet}(W_0).
\end{equation*}

In the present paper, we will be mostly interested in the following version, which is based on the construction of $\HMtil$ in \cite{Blo}. For a given basepoint $p\in Y$, we can define the relatively graded $\ztwo[Q]/Q^2$-module $\HMItil_*(Y,p)$, which fits in the exact triangle of $\ztwo[Q]/Q^2$-modules
\begin{center}
\begin{tikzpicture}
\matrix (m) [matrix of math nodes,row sep=1em,column sep=1em,minimum width=2em]
  {
  \HMIt_{*}(Y) && \HMIt_{*}(Y)\\
  &\HMItil_{*}(Y,p) &\\};
  \path[-stealth]
  (m-1-1) edge node [above]{$\cdot U$} (m-1-3)
  (m-2-2) edge node [left]{} (m-1-1)
  (m-1-3) edge node [right]{} (m-2-2)  
  ;
\end{tikzpicture}
\end{center}
and
\begin{center}\label{triangle}
\begin{tikzpicture}
\matrix (m) [matrix of math nodes,row sep=1em,column sep=1em,minimum width=2em]
  {
  \HMtil_{*}(Y,p) && \HMtil_{*}(Y,p)[-1]\\
  &\HMItil_{*}(Y,p) &\\};
  \path[-stealth]
  (m-1-1) edge node [above]{$1+\jmath_*$} (m-1-3)
  (m-2-2) edge node [left]{} (m-1-1)
  (m-1-3) edge node [right]{} (m-2-2)  
  ;
\end{tikzpicture}
\end{center}
As the versions discussed before, this invariant decomposes along equivalence classes in $\mathrm{Spin}^c(Y)/\jmath$, and the analogous identification of (\ref{nonspin}) holds. We have that
\begin{equation*}
\HMItil_*(S^3,p)=\Br,
\end{equation*}
the element $Q$ being in degree $0$. Given a cobordism $(W,\gamma)$ from $(Y_0,\gamma_0)$ to $(Y_1,\gamma_1)$, where $\gamma$ is a properly embedded arc in $W$ from $p_0$ to $p_1$, one can define a map
\begin{equation*}
\HMItil_*(W,\gamma):\HMItil_*(Y_0,p_0)\rightarrow \HMItil_*(Y_1,p_1).
\end{equation*}
This assignment is functorial (in the category of based three-manifolds and cobordisms with an embedded arc), and is compatible with the maps in the \textit{to} version and the triangle above.

\vspace{0.5cm}
\textbf{The topology of the space of operators with simple spectrum. }We discuss a simple result which underlines the construction of the invariants. As a warmup, consider $i\mathfrak{u}(2)$, the space four-dimensional space of self-adjoint operators on $\mathbb{C}^2$. The only elements with non-simple spectrum are the real multiples of the identity. This implies that the space of self-adjoint operators with simple spectrum on $\mathbb{C}^2$ is connected and simply connected (while it has non-trivial higher homotopy groups), as the locus of operators with non-simple spectrum has codimension three.
\par
In the setting we are interested, let $S\rightarrow Y$ be a spinor bundle on a three-manifold. Following Chapter $12$ of \cite{KM}, we consider the space $\mathrm{Op}^{sa}$ of self-adjoint Fredholm operators of the form
\begin{equation*}
D_{B}+h:L^2_k(Y;S)\rightarrow L^2_{k-1}(Y;S)
\end{equation*}
where $D_{B}$ is the Dirac operator associated to a smooth connection $B$ on $S$ and $h$ is a self-adjoint operator that extends to a bounded operator form $L^2_j(Y;S)$ to itself, for $j\leq k$ (here $k$ is chosen to be large enough). These is the kind of operators that arises when analyzing the reducible solutions of the Seiberg-Witten equations. In particular, before perturbing the equations, a reducible solution $[B,0]$ gives rise to a family of solutions on the blow up corresponding to the projectivizations of the eigenspaces of the Dirac operator $D_B$.
\begin{lemma}\label{sc}
The space of operators in $\mathrm{Op}^{sa}$ with simple spectum is connected and simply connected.
\end{lemma} 
\begin{proof}
This boils down to the example above via the following observation from Section $12.6$ of \cite{KM}. The space of operators whose spectrum is not simple is a countable union of Fredholm maps $F_n$ defined as follows. Take the domain of $F_n$ to be $\mathrm{Op}^{sa}_n\times \mathbb{R}$, where $\mathrm{Op}^{sa}_n\subset \mathrm{Op}^{sa}$ is the subspace of operators for which $0$ is an eigenvalue of multiplicity exactly $n$, and $F_n$ sends $(L,\lambda)$ to $L+\lambda$. This is a local embedding, and the normal bundle to its image at $L+\lambda$ is given by the space of traceless self-adjoint endomorphisms of $\mathrm{Ker}L$. In particular, each of the Fredholm maps has index at least $-3$, and the result readily follows from a standard transversality argument.
\end{proof}

\begin{remark}A case of particular interest for us is that of $S^3$ equipped with a metric of positive scalar curvature, e.g. the round one. In this case, for a small pertubation $\hat{\q}$ the perturbed Dirac operator $D_{B_0,\hat{\q}}$ associated to the trivial connection $B_0$ does not have kernel (see Section $22.7$ in \cite{KM}). Hence such a perturbation is regular if and only if the associated Dirac operator has simple spectrum. Therefore the lemma above can rephrased as follows: the space of regular small perturbations on $S^3$ is connected and simply connected.
\par
We will also use the following consequence of simply connectedness: any two paths of small regular perturbations $\gamma_0$ and $\gamma_1$ are homotopic, and any two such homotopies $\gamma_t$ and $\gamma'_t$ are homotopic relative to be boundary paths $\gamma_0,\gamma_1$.
\end{remark}

\vspace{0.5cm}
\textbf{The $U$-map in monopole Floer homology. }We quickly review the definition of the $U$-map in monopole Floer homology and $\HMtil(Y,p)$. To a given three-manifold $Y$ equipped with a Riemannian metric (which we generally drop from the notation), a spin$^c$ structure $\spin$ and a regular perturbation $\q$, we can associate the Floer chain complex $(\check{C}_*(Y,\spin,\q),\check{\partial})$. For two different choices of metrics and perturbations, the associated homology groups, denoted by $\HMt(Y,\spin,\q)$, differ by a canonical isomorphism induced by a continuation map.
\par
Suppose now we are given a basepoint $p\in Y$. Fix the round metric on $S^3$, and a small regular perturbation $\hat{\q}$, so that there are no irreducible critical points. From this data, we can construct a chain map
\begin{equation*}
\check{U}_{\hat{\q}}: \check{C}_*(Y,\spin,\q)\rightarrow \check{C}_*(Y,\spin,\q)
\end{equation*}
defining the $\ztwo[U]$-module structure on homology. This chain map is defined by counting trajectories on $[0,1]\times Y\setminus \mathrm{nbhd}(\{1/2\}\times\{p\})$ with cylindrical ends attached which converge on the $S^3$ end (which we think of as an incoming end, with perturbation $\hat{\q}$) to the critical point corresponding to the second negative eigenvalue of the Dirac operator. We can then define as in \cite{Blo} the chain complex 
\begin{align*}
\Ctil_*(Y,p,\spin,\q,\hat{\q})&=\check{C}(Y,\spin,\q)\oplus \check{C}(Y,\spin,\q)[1]\\
\tilde{\partial}&=\left(\begin{matrix}
\check{\partial} & 0\\
\check{U}_{\hat{\q}} & \check{\partial}
\end{matrix}\right).
\end{align*}
The homology is again independent up to canonical isomorphism of the choices of perturbations, and we call the result $\HMtil_*(Y,p,\spin)$. There is a natural two-step filtration, inducing the exact triangle
\begin{center}
\begin{tikzpicture}
\matrix (m) [matrix of math nodes,row sep=1em,column sep=1em,minimum width=2em]
  {
  \HMt_{*}(Y,\spin,\q) && \HMt_{*}(Y,\spin,\q)\\
  &\HMtil_{*}(Y,p,\spin,\q,\hat{\q}) &\\};
  \path[-stealth]
  (m-1-1) edge node [above]{$\cdot U$} (m-1-3)
  (m-2-2) edge node [left]{} (m-1-1)
  (m-1-3) edge node [right]{} (m-2-2)  
  ;
\end{tikzpicture}
\end{center}
We discuss in detail the independence with respect to the choice of perturbations. Suppose that $\q'$ and $\hat{\q}'$ are different choices of perturbations on $S^3$. We can choose a two paths of perturbations connecting our previous perturbations to the new ones. Furthermore, we can choose the path of small perturbations on $S^3$ such that at each time the Dirac operator has simple spectrum, so that each perturbation in the path is regular (see Lemma \ref{sc} and following remark). Now, using these families of perturbations we can constructs two maps
\begin{equation*}
\check{\Phi}, \check{K}:\check{C}_*(Y,\spin,\q)\rightarrow \check{C}(Y,\spin,\q')
\end{equation*}
such that
\begin{align*}
\check{\partial}\circ\check{\Phi}+\check{\Phi}\circ\check{\partial}&=0\\
\check{\partial}\circ\check{K}+\check{K}\circ\check{\partial}&=\check{U}_{\hat{\q}'}\circ \check{\Phi}+\check{\Phi}\circ\check{U}_{\hat{\q}}.
\end{align*}
While the first map is simply the continuation map inducing the canonical isomorphism in homology, the second map uses parametrized moduli spaces on a punctured $I\times Y$ where the perturbation on the additional incoming $S^3$ varies along the pat from $\hat{\q}$ to $\hat{\q}'$ chosen above. The key point is that at each time the perturbation is regular, so one can look at moduli spaces of solutions parametrized by the family of perturbations which converge on this end to the second unstable critical point. The map $\check{K}$ is defined by a suitable count of solutions in these moduli spaces, using the same formulas as in Chapter $25$ in \cite{KM}. It is worth pointing out that we are using that the various critical points have degrees differing by at least two, so that the $S^3$ does not contribute to the formulas when looking at codimension one strata.  Furthermore, while relevant analytical details for this moduli spaces of this kind of are not explicitly stated in \cite{KM}, the results there generalize without any significant modification because of the assumption that the perturbation on $S^3$ is regular at each point on the path. Finally, we can combine these maps into
\begin{equation*}
\tilde{\Phi}=\left(\begin{matrix}
\check{\Phi} & 0\\
\check{K} & \check{\Phi}
\end{matrix}\right): \Ctil_*(Y,p,\spin,\q,\hat{\q})\rightarrow \Ctil_*(Y,p,\spin,\q',\hat{\q}').
\end{equation*}
This map induces an isomorphism in cohomology: the two-step filtration on the complexes induces a spectral sequence, and the map induced on the $E^1$-page is an isomorphism (because the entries are continuation maps). Furthermore, two such maps are always chain homotopic because the space small regular perturbations is simply connected, so the induced isomorphism is canonical.
\par
More in general, we recall how the map induced by a cobordism together with a path $(W,\gamma)$ is defined. For For some choice of regular perturbation $\p$ one can define the chain map
\begin{equation*}
\check{m}(W,\spin,\p): \check{C}_*(Y_0,\spin_0,\q_0)\rightarrow\check{C}_*(Y_1,\spin_1,\q_1).
\end{equation*}
Using the path $\gamma$, one can define (after making additional choices, including a path from $\hat{\q}_0$ to $\hat{\q}_1$) a chain homotopy $\check{H}$ such that
\begin{equation*}
\check{\partial}\circ\check{H}+\check{H}\circ\check{\partial}=\check{U}_{\hat{\q}_1}\circ \check{m}(W,\spin,\p)+\check{m}(W,\spin,\p)\circ\check{U}_{\hat{\q}_0}.
\end{equation*}
We then define the chain map
\begin{equation*}
\tilde{m}(W,\spin,\gamma,\p)=\left(\begin{matrix}
\check{m}(W,\spin,\p) & 0\\
\check{H} & \check{m}(W,\spin,\p)
\end{matrix}\right): \check{C}_*(Y_0,p_0,\spin_0,\q_0,\hat{q}_0)\rightarrow\check{C}_*(Y_1,p_1,\spin_1,\q_1,\hat{\q}_1),
\end{equation*}
and denote the map induced in homology by $\tilde{m}(W,\spin,\p)$. Its functoriality properties follow from neck-stretching arguments, see Chapter $26$ in \cite{KM}.

\vspace{0.5cm}
\textbf{Involutive monopole Floer homology. } We finally discuss the definition and properties of involutive monopole Floer homology. We focus on the \textit{tilde} version both because it is the one we will be working with in the rest of the paper, and because the construction for the other cases is essentially simpler. The key idea is that there is a canonical identification of the space of configurations
\begin{equation*}
\jmath: \mathcal{C}(Y,\spin)\rightarrow \mathcal{C}(Y,\bar{\spin})
\end{equation*}
given by $\jmath(B,\Psi)=(\bar{B},\Psi\cdot j)$. Here $\bar{B}$ denotes the conjugate connection, and for the multiplication by $j$ we identify the spinor representation of $\mathbb{C}^2$ with $\mathbb{H}$. This identification squares to $-1$, so that we get an identification
\begin{equation*}
\jmath: \mathcal{B}(Y,\spin)\rightarrow \mathcal{B}(Y,\bar{\spin})
\end{equation*}
which squares to the identity (as $-1$ is a gauge transformation). Via this identifications, a regular perturbation $\q$ for $(Y,\spin)$ induces a regular perturbation $\jmath_*\q$ for $(Y,\bar{\spin})$, and we have induced identification of chain complexes
\begin{equation*}
\tilde{\jmath}_*:\Ctil(Y,p,\spin,\q,\hat{\q})\rightarrow \Ctil(Y,p,\bar{\spin},\jmath_*\q,\jmath_*\hat{\q}).
\end{equation*}
Suppose now that $\spin$ is self-conjugate. There is a quasi-isomorphisms
\begin{equation*}
\tilde{\Phi}:\Ctil(Y,p,\bar{\spin},\jmath_*\q,\jmath_*\hat{\q})\rightarrow \Ctil(Y,p,\spin,\q,\hat{\q}),
\end{equation*}
given by the continuation map. Such a map depends on a choice of a regular homotopy $k$ from $(\jmath_*\q,\jmath_*\hat{\q})$ and $(\q,\hat{\q})$. We can then define the chain map
\begin{equation*}
\tilde{\iota}:=\tilde{\Phi}\circ\tilde{\jmath}_*:\Ctil(Y,p,\spin,\q,\hat{\q})\rightarrow \Ctil(Y,p,\spin,\q,\hat{\q})
\end{equation*}
and define the involutive Floer chain complex $\widetilde{CI}(Y,p,\spin,\q,\hat{\q},k)$ to be the mapping cone
\begin{equation*}
\Ctil(Y,p,\spin,\q,\hat{\q})\stackrel{Q(1+\tilde{\iota})}{\longrightarrow}Q\cdot\Ctil(Y,p,\spin,\q,\hat{\q})
\end{equation*}
where $Q$ is a formal variable of degree $-1$. We denote its homology by $\HMItil(Y,p,\spin,\q,\hat{\q})$. This chain complex is naturally a graded module over $\ztwo[Q]/Q^2$, and it comes with a natural filtration, which we call the $Q$-filtration, inducing the exact triangle (\ref{triangle}).
\par
We can define the map on $\HMItil$ induced by a cobordism $(W,\gamma)$ equipped with a self-conjugate spin$^c$ structure $\spin$ from $(Y_0,\spin_0)$ to $(Y_1,\spin_1)$ as follows. Consider a regular perturbation $\mathfrak{p}$ on $W$, so that one obtains chain maps
\begin{equation*}
\tilde{m}(W,\spin,\gamma,\p):\Ctil(Y_0,p,\spin_0,\q_0,\hat{\q}_0,k_0)\rightarrow\Ctil(Y_1,p_1,\spin_1,\q_1,\hat{\q}_1,k_1)
\end{equation*}
inducing the cobordism map. Applying the involution, one obtains a chain map
\begin{equation*}
\tilde{m}(W,\gamma,\spin,\jmath_*\p):\Ctil(Y_0,p_0,\spin_0,\jmath_*\q_0,\jmath_*\hat{\q}_0,\jmath_*k_0)\rightarrow\Ctil(Y_1,p_1,\spin_1,\jmath_*\q_1,\jmath_*\hat{\q}_1,\jmath_*k_1)
\end{equation*}
computing (after canonical identifications) the same map in homology. Indeed, for suitable choice of data one can define a map $\tilde{h}$ such that
\begin{equation*}
\tilde{\partial}\circ\tilde{h}+\tilde{h}\circ\tilde{\partial}=\tilde{m}(W,\gamma,\spin,\p)\circ\tilde{\Phi}_0+\tilde{\Phi}_1\circ \tilde{m}(W,\gamma,\spin,\jmath_*\p)
\end{equation*}
there the $\tilde{\Phi}_i$ are the continuation maps used in the definitions of $\widetilde{CI}$ for $Y_i$. Again in this choice we need the perturbation on $S^3$ to be regular at each point, and this is again possible in light of Lemma \ref{sc} and the following remark. We can combine these in order to define the chain map
\begin{equation*}
\widetilde{mi}(W,\gamma,\p,h)=\left(\begin{matrix}
\tilde{m}(W,\gamma,\spin,\p) & 0\\
\tilde{h}\circ\tilde{\jmath}_* & \tilde{m}(W,\gamma,\spin,\p)
\end{matrix}\right):\widetilde{CI}(Y_0,p_0,\spin_0,\q_0,\hat{\q}_0,k_0)\rightarrow \widetilde{CI}(Y_1,p_1,\spin_1,\q_1,\hat{\q}_1,k_1).
\end{equation*}
We denote the map induced in homology by   $\HMIt(W,\gamma,\p,h)$. We state the relevant invariance result.
\begin{prop}
For any choice of two choices of perturbations and homotopies there exists a canonical isomorphism
\begin{equation*}
\widetilde{\Phi I}((\q,\hat{\q},k),(\q',\hat{\q}',k')):\HMItil_*(Y,p,\spin,\q,\hat{\q},k)\rightarrow \HMItil_*(Y,p,\spin,\q',\hat{\q}',k')
\end{equation*}
such that the following hold:
\begin{enumerate}
\item $\widetilde{\Phi I}((\q,\hat{\q},k),(\q,\hat{\q},k))$ is the identity of $\HMIt_*((\q,\hat{\q},k))$;
\item we have $\widetilde{\Phi I}((\q',\hat{\q}',k'),(\q'',\hat{\q}'',k''))\circ\widetilde{\Phi I}((\q,\hat{\q},k),(\q',\hat{\q}',k'))=\widetilde{\Phi I}((\q,\hat{\q},k),(\q'',\hat{\q}'',k''))$.
\end{enumerate}
In particular, there is a well defined involutive Floer homology group $\HMIt_*(Y,p,\spin)$.
\end{prop}

The proof of this results follows very closely those of the counterparts in usual monopole Floer homology (Chapter $23$ of \cite{KM}). Of course, the map $\widetilde{\Phi I}$ arises as a continuation map. Its definition requires the choice of a regular homotopy $k_t$ between $k$ and $k'$, which are themselves homotopies from $(\q,\hat{\q})$ to $(\jmath_*\q,\jmath_*\hat{\q})$ and from $(\q',\hat{\q}')$ to $(\jmath_*\q',\jmath_*\hat{\q}')$ respectively. On the other hand, any two such homotopies $k_t$ are homotopic relative to the boundary because of Lemma \ref{sc} and following remark, so that the map induced in homology can be shown to be independent of the choice as in the usual case. In particular, this also implies the property $(1)$ in the statement. Property $(2)$ (and more in general the composition property for cobordisms) follows from a standard neck stretching argument. 

\begin{example}In the case of $S^3$, we have that $\HMtil_*(S^3,p)$ is $\ztwo$, and indeed the generator can be identified (for some small regular perturbation, which we omit from the notation) in the chain complex
\begin{equation*}
\check{C}_*(S^3)\stackrel{\check{U}}{\longrightarrow} \check{C}_*(S^3)
\end{equation*}
with the bottom generator in the left summand. From this the identification of $\HMItil(S^3,p)$ with $\Br$ readily follows, as $\jmath_*$ induces the identity in isomorphism.
\end{example}

So far we have only discussed the case of a self-conjugate spin$^c$ structure, but one should include in the discussion also the non self-conjugate ones in order to achieve full functoriality: indeed, in general the composition of spin cobordisms does not need to be spin. In the case on a pair of distinct spin$^c$ structures $\spin\neq\bar{\spin}$ we can define $\check{C}I(Y,[\spin])$ to be
\begin{equation*}
\tilde{C}(Y,\spin,\q,\hat{\q})\oplus \tilde{C}(Y,\bar{\spin},\jmath_*\q,\jmath_*\hat{\q})\stackrel{1+\jmath_*}{\longrightarrow}Q\cdot(\tilde{C}(Y,\spin,\q,\hat{\q})\oplus \tilde{C}(Y,\bar{\spin},\jmath_*\q,\jmath_*\hat{\q}))
\end{equation*}
and denote its homology by and $\HMIt(Y,p,[\spin],[\q],[\hat{\q}])$. The invariance proof discussed above works without complications. On the other hand, as $\jmath$ acts freely on the complex exchanging the two summands, the homology is naturally identified as a $\ztwo[Q]/Q^2$-module with
\begin{equation*}
\HMtil(Y,p,\spin)\oplus\HMtil(Y,p,\spin)[-1]\equiv \HMtil(Y,p,\bar{\spin})\oplus \HMtil(Y,p,\bar{\spin})[-1],
\end{equation*}
where the action by $Q$ is trivial. The definition of the maps induced by a cobordism $(W,\gamma)$ equipped with a pair of conjugate non-isomorphic spin$^c$ structures $\spin\neq \bar{\spin}$ is straightforward at this point, and we spell it out in detail in the case in which the restriction to the incoming end is not self-conjugate while the restriction to the outgoing one is. In this case, for a fixed perturbation $\p,\bar{\p}$ for the spin$^c$ structures $\spin$ and $\bar{\spin}$ (and suitable chain homotopies for the $U$ actions), one can define the chain maps
\begin{align*}
\tilde{m}(W,\spin,\p)&:\tilde{C}(Y_0,\spin_0,\q_0,\hat{\q}_0)\rightarrow\tilde{C}(Y_1,\spin_1,\q_1,\hat{\q}_1)\\
\tilde{m}(W,\bar{\spin},\bar{\p})&:\tilde{C}(Y_0,\bar{\spin}_0,\jmath_*\q_0,\jmath_*\hat{\q}_0 )\rightarrow \tilde{C}(Y_1,\spin_1,\q_1,\hat{\q}_1)
\end{align*}
defining the usual map induced by the cobordisms. One can then define the chain map
\begin{equation*}
\check{m}i(W,\p,\bar{\p}):\widetilde{CI}(Y_0,[\spin_0],[\q_0],[\hat{\q}_0])\rightarrow\widetilde{CI}(Y_1,\spin_1,\q_1,\hat{\q}_1, k)
\end{equation*}
given by the matrix
\begin{equation*}
\left(\begin{matrix}
\tilde{m}(W,\spin,\p)+\tilde{m}(W,\bar{\spin},\bar{\p}) & 0\\
h & \tilde{\Phi}\circ(\tilde{m}(W,\bar{\spin},{\jmath}_*\p)+\tilde{m}(W,{\spin},\jmath_*{\bar{\p}}))
\end{matrix}\right).
\end{equation*}
where $h$ is a suitable additional term involving a certain chain homotopy for the $U$ action and the map $\tilde{\Phi}$. The induced map is independent of the choices, and we denote it by $\HMIt_*(W,[\spin])$.

\vspace{1cm}
\section{The link surgery spectral sequence}\label{surgery}
While the results (which are analogous to those in \cite{OSBr} and \cite{Blo}) in the present section hold for all the versions of involutive Floer homology, we focus on the case of $\HMItil$ as it will be important in the next section. Suppose that we are given a link $L$ with $l$ components $K_i$ in a based three-manifold $(Y,p)$ so that the basepoint $p$ does not belong to $L$. For each component $K_i$ fix a surgery triple $\gamma_{0}^i,\gamma_{1}^i,\gamma_{\infty}^i$, i.e. a collection of simple closed curves in $\partial(\mathrm{nbhd}(K_i))$ such that
\begin{equation*}
\gamma_{0}^i\cdot\gamma_{1}^i=\gamma_{1}^i\cdot\gamma_{\infty}^i=\gamma_{\infty}^i\cdot\gamma_{0}^i=-1.
\end{equation*}
For each $I\in\{0,1\}^l$, let $Y_I$ be the manifold obtained from $Y$ by surgery along $K_i$ according to the $i$th component of $I$, for each $i$. Also, we can define $w(I)$ to be the number of ones in $I$, and consider the lexicographical order on $\{0,1\}^k$. When $I\leq J$ and $w(J)-w(I)=1$, there is a natural $2$-handle attachment from $Y_I$ to $Y_J$ which we denote by $W_{IJ}$. The main result is then the following.

\begin{thm}\label{linkSS}
There exists a spectral sequence of $\Br$-modules whose $E^1$ page is
\begin{align*}
E^1&=\bigoplus_{I\in \{0,1\}^l} \HMItil_*(Y_I,p)\\
d^1&=\sum_{I\leq J,w(J)-w(I)=1} \HMItil_*(W_{IJ},[0,1]\times \{p\}).
\end{align*}
and converges to $\HMItil_*(Y_{\mathbf{\infty}},p)$, where $Y_{\mathbf{\infty}}$ is the manifold by surgery along all the $\gamma_{\infty}^i$. Furthermore, the spectral sequence is independent (up to canonical isomorphism) of the choice of data.
\end{thm}
\begin{remark}
Notice that $p$ determines a point in each of the $Y_I$. For the maps in the $d^1$ differential, the path we consider is the product path $[0,1]\times \{p\}$, as we can identify the cobordism $W_{IJ}$ as the product cobordism $[0,1]\times Y_I$ with a two-handle attached far away from the knot.
\end{remark}

The following result is one of the key reasons why we are working in the involutive setting rather than the $\Pin$-one.
\begin{prop}
For any cobordism $W$, the map $\HMItil_*(W\hash \overline{\mathbb{C}P}^2)$ is zero. Indeed, there is an explicit chain homotopy to zero.
\end{prop}
The analogous result in the $\Pin$-case states that the maps is generally non-zero, with image contained in the image of $Q^2$ (see \cite{Lin2}). We expected our result to hold as $Q^2$ is zero in the involutive setting. 
\begin{proof}
The proof follows as in the usual case \cite{KMOS} by a neck stretching argument. Fix the round metric on $S^3$ and a small regular perturbation $\hat{\q}$ on $S^3$. Consider for $S\in[0,\infty)$ the manifold
\begin{equation*}
W(S)=(W^*\setminus D^4)\cup [0,S]\times S^3\cup (\overline{\mathbb{C}P}^2\setminus D^4),
\end{equation*}
equipped with a metric which is independent of $S$ on the first and third term and is a product one on the middle term. We can look at the moduli spaces of solutions on this cobordism parametrized by $S$. In order to compactify it, we need to also consider the pairs of solutions $(\gamma_W,\gamma_{\overline{\mathbb{C}P}^2})$ on the disconnected cobordism
\begin{equation*}
W(\infty)= (W\setminus D^4)^*\amalg  (\overline{\mathbb{C}P}^2\setminus D^4)^*
\end{equation*}
that are asymptotic on the puncture to the same critical point. When the moduli space is zero dimensional, the solutions come in pairs
\begin{equation*}
(\gamma_W,\gamma_{\overline{\mathbb{C}P}^2})\quad \text{and}\quad (\gamma_W,\bar{\gamma}_{\overline{\mathbb{C}P}^2}).
\end{equation*}
This follows from the fact that the spin$^c$ structures on $\overline{\mathbb{C}P}^2$ come in conjugate pairs and when the moduli space on $(\overline{\mathbb{C}P}^2\setminus D^4)^*$ is zero dimensional, it consists of a single point (for a given choice of perturbations, see Lemma $27.4.2$ in \cite{KM}). In particular, their total contribution is zero, so that considering zero dimensional parametrized moduli spaces on $W(S)$ one obtains a chain homotopy from the original map induced by $W$ and zero.
\par
The same argument works with little modification in the involutive case. We can first pick a homotopy through small regular perturbations rom $\hat{\q}$ to $\jmath_*\hat{\q}$. We can then run the neck stretching argument to define a homotopy between $\widetilde{mi}(W)$ and the map induced by looking at the cobordism in which the neck is stretched to infinity. While the perturbation on the puncture changes along the path of homotopies, at each time the zero dimensional moduli spaces on $\overline{\mathbb{C}P}^2$ corresponding to conjugate spin$^c$ structures provide two canceling solutions as above, so the total contribution is zero.
\end{proof}

The key point in the proof is that in order to show that the map vanishes we only need the space of small regular perturbations to be simply connected. We can interpret the fact that the analogous result in the $\Pin$-theory is false from this topological perspective: in order to define the $\Pin$-theory, one would need higher homotopies (see \cite{SS}), and the non-vanishing of the map corresponds to the non-triviality of the higher homotopy groups of the space of regular perturbations.
\begin{prop}
Given a knot $K\subset Y$ and a surgery triple $\gamma_i$ for $i=1,2,3$, denote by $Y_i$ the manifold obtained by surgery on $K$ with framing $\gamma_i$. Then there is an exact triangle
\begin{center}
\begin{tikzpicture}
\matrix (m) [matrix of math nodes,row sep=1em,column sep=1em,minimum width=2em]
  {
  \HMItil_{*}(Y_1,p) && \HMItil_{*}(Y_2,p)\\
  &\HMItil_{*}(Y_3,p) &\\};
  \path[-stealth]
  (m-1-1) edge node [above]{} (m-1-3)
  (m-2-2) edge node [left]{} (m-1-1)
  (m-1-3) edge node [right]{} (m-2-2)  
  ;
\end{tikzpicture}
\end{center}
where in the latter we chose the basepoint $p\not\in K$, and the maps are the maps induced by the corresponding $2$-handle attachment cobordisms.
\end{prop}
\begin{proof}
The result follows in the same manner as the triangle in monopole Floer homology, see \cite{KMOS}. We spell out the few additional details to be checked. Denote by $W_i$ the $2$-handle attachment from $Y_i$ to $Y_{i+1}$, and consider the induced map
\begin{equation*}
\tilde{f}_i: \widetilde{CI}_*(Y_i)\rightarrow \widetilde{CI}_*(Y_{i+1}).
\end{equation*}
We will supress the various basepoints and additional choices for simplicity. The composite cobordism $W_{i+1}\circ W_i$ is naturally identified with $\overline{W}_{i+2}\hash \overline{\mathbb{C}P}^2$ and the previous discussion constructs chain null-homotopies $\tilde{H}_i$ for $\tilde{f}_{i+1}\circ\tilde{f}_i$ (this is associated to the stretching along the two hypersurfaces $Y_{i+1}$ and $S_i$, see Figure \ref{triplecomp}). The result then follows by general homological algebra (see the triangle detection lemma $5.1$ in \cite{KMOS}) once we show that
\begin{equation}\label{goal}
\tilde{H}_{i+1}\circ \tilde{f}_i+\tilde{f}_{i+2}\circ\tilde{H}_i: \widetilde{CI}_*(Y_i)\rightarrow \widetilde{CI}(Y_i)
\end{equation}
is a quasi-isomorphism. To see why this holds, one considers the triple composite cobordism $W_{i+2}\circ W_{i+1}\circ W_i$. This contains five distinguished hypersurfaces (see Figure \ref{triplecomp}), and one considers a two-dimensional family of (possibly degenerate) metrics and perturbations parametrized by the pentagon $\bar{P}$ in Figure \ref{pentagon}. The latter is obtained as the union of five squares, each of which is a two dimensional family of metrics and perturbations corresponding to disjoint hypersurfaces. The edges $Q(S_i)$ correspond to stretching out to infinity a $\overline{\mathbb{C}P}^2$ summand along the hypersurface $S_i$, so that when considering zero dimensional moduli space the corresponding edges contribute zero as in the previous proposition. Again, we only need here that the space of small regular perturbations is simply connected. Hence counting zero dimensional moduli spaces parametrized by this pentagon provides a chain homotopy between (\ref{goal}) (which corresponds to the hypersurfaces $Y_{i+1}$ and $Y_{i+2}$ stretched out to infinity, hence the edges $Q(Y_{i+1})$ and $Q(Y_{i+2})$) and a map $\tilde{L}$ obtained by counting certain solutions on the cobordism with the hypersurface $R$ stretched to infinity. The latter is a copy of $S^2\times S^1$ with a metric of positive scalar curvature, and it is shown in \cite{KMOS} that in the usual setting the induced map is a multiplication by a power series in $U$ with leading coefficient $1$ (hence an isomorphism). To show that the same statement holds in the involutive setting, we first need to show that for the space regular of perturbations of $S^2\times S^1$ we consider is simply connected, so that all the constructions can be carried over. The perturbations used in \cite{KMOS} have the following form. First, one can fix a $\jmath$-equivariant retraction
\begin{equation*}
\mathcal{B}(S^2\times S^1, \spin_0)\rightarrow \mathbb{T}
\end{equation*}
where the latter is the circle of flat connections (on which $\jmath$ acts by conjugation). Consider a $\jmath$-invariant Morse function
\begin{equation*}
f:\mathbb{T}\rightarrow \mathbb{R}
\end{equation*}
which only has two critical points (e.g. the standard height function). We can then achieve transversality by considering perturbations of the form
\begin{equation}\label{pertS2S1}
\varepsilon f\circ p +\hat{\q}
\end{equation}
where $\varepsilon$ is fixed small so that no irreducible critical points are introduced and $\hat{\q}$ is small compared to epsilon. There will be only two critical points for the perturbed Chern-Simons-Dirac functional, namely the reducibles corresponding to the two spin connections on $S^2\times S^1$, and transversality is equivalent to the corresponding Dirac operators having simple spectrum. Now under the involution this pushes forward to $\varepsilon f\circ p +\jmath_*\hat\q$, hence the situation is the same as that of $S^3$ and simply connectivity follows from Lemma \ref{sc} and the following remark.
\par
This shows in the involutive setting that the pentagon of metrics $\bar{P}$ together with a suitable regular homotopy to its conjugate induces a chain homotopy from (\ref{goal}) to
\begin{equation*}
\widetilde{L}:\widetilde{CI}_*(Y_i)\rightarrow \widetilde{CI}(Y_i).
\end{equation*}
The latter is defined counting contributions from $R$ with the same formulas as \cite{KMOS}. Taking the $Q$-filtration, we see that the induced map of the $E^1$ page is an isomorphism because of the result in the usual case, and the proposition follows.
\end{proof}

\begin{figure}
  \centering
\def\svgwidth{0.7\textwidth}
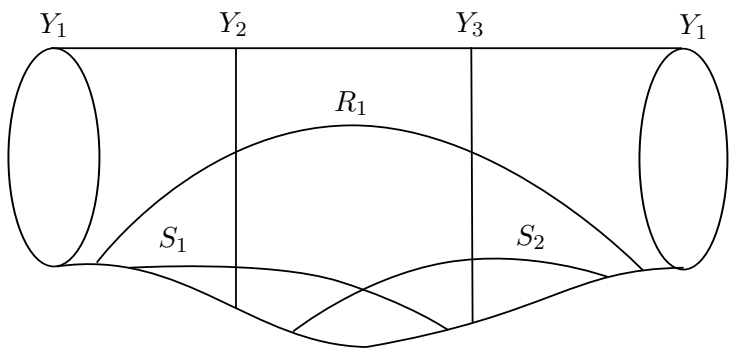
    \caption{The five hypersurfaces in the triple composite.}\label{triplecomp}
\end{figure}

\begin{figure}
  \centering
\def\svgwidth{0.5\textwidth}
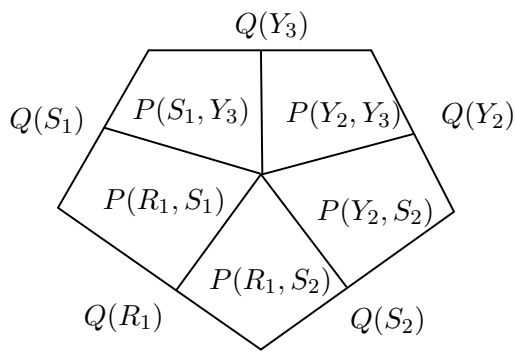
    \caption{The family of metrics $\bar{P}$ Each square correspond to stretching along two disjoint surfaces.}\label{pentagon}
\end{figure}

We are finally ready to prove Theorem \ref{linkSS}. The proof is essentially the same as that provided in \cite{Blo} (to which we refer for the details), and we discuss discuss which adaptations are needed for the involutive case. Define $X$ to be the chain complex
\begin{equation*}
X= \bigoplus_{I\in \{0,1\}^l} \widetilde{CI}_*(Y_I)
\end{equation*}
where the differential $D$ is obtained by summing over a collection of maps
\begin{equation*}
D^I_J:\widetilde{CI}_*(Y_I)\rightarrow \widetilde{CI}_*(Y_J)
\end{equation*}
for $I\leq J$. To each $I\leq J$, there is an associated $(w(J)-w(I)-1)$-dimensional permutohedron $P_{IJ}$, and the maps are defined by considering the moduli spaces parametrized by a corresponding family of metrics and perturbations, together with a regular homotopy to its image under the pushforward along $\jmath_*$. These families correspond to stretching along the collection of $2^{(w(J)-w(I)-1)}-2$ hypersurfaces $Y_K$ for $I<K<J$. For example, map $D^I_I$ is simply the differential of $\widetilde{CI}(Y_I)$, while for $w(J)-w(I)=1$ this is just the map induced by the corresponding $2$-handle attachment. The proof of \cite{Blo} carries over to show that this is indeed a chain complex. The filtration by $w$ induces a spectral sequence whose $E^1$ page is readily identified as in the statement.
\par
We then need to prove that the homology of $X$ is $\HMItil_*(Y_{\infty})$. We proceed by induction on the number of components of the link as in \cite{Blo}. We slightly generalize the notation $Y_I$ to denote surgeries with coefficients in $I\in\{0,1,\infty\}^k$. Consider the component $K_k$. We have that the construction provided in \cite{Blo} of the chain complex
\begin{equation*}
\tilde{X}^k=\bigoplus_{I\in \{0,1\}^{k-1}\times\{0,1,\infty\}\times \{\infty\}^{l-}k} \widetilde{CI}_*(Y_I)
\end{equation*}
associated to the lattice $\{0,1\}^{l-1}\times\{0,1,\infty\}\times\{\infty\}^k$ (where in the $k$th component $0<1<\infty$) carries over to the involutive case in the same fashion as the construction of $X$. Indeed, his construction involves families of metrics and perturbations parametrized by certain graph-associahedra corresponding to stretching along hypersurfaces which are copies of either $Y_K$ for $I<K<J$, $S^3$ or $S^2\times S^1$. For the latter two, we require as before that the homotopy to the conjugate is through regular perturbations.  The proof of the surgery exact triangle for $K_k$ can be rephrased to say that the complex $\tilde{X}^k$ is acyclic. If we consider the $2$-step filtration on $X^l$ obtained by considering the inclusion $\{\infty\}\subset\{0,1,\infty\}$ in the $k$th coordinate, the fact that the chain complex is acyclic tells us that there is a quasi-isomorphism
\begin{equation*}
H_*\left(\bigoplus_{I\in \{0,1\}^{k-1}\times\{\infty\}^{l-k+1}} \widetilde{CI}_*(Y_I)\right)\cong H_*(X),
\end{equation*}
and the left-hand side is quasi-isomorphic to $\HMItil_*(Y_{\infty})$ by inductive hypothesis.

\vspace{1cm}
\section{Relation with Bar-Natan's theory}\label{barnatan}

In this section we discuss how Bar-Natan's theory, which is the modification of Khovanov's construction obtained via the operations (\ref{BN}), naturally arises in our context after setting $u^2$ to be zero. Recall that we will deal with the reduced version $\BN^2$: after choosing a basepoint, this is the homology of the complex obtained by quotienting the $\Br$-subcomplex consisting of elements which have $\vm$ in the entry corresponding to the basepoint. The same change of basis adopted in \cite{ORS} shows that this is independent of the choice of the basepoint, and furthermore we have the splitting of bigraded $\Br$-modules
\begin{equation*}
\mathrm{BN}^2_{*,*}(L)=\BN^2_{*,*}(L)\oplus \BN^2_{*-2,*}(L)
\end{equation*}
Consider now a diagram $D$ of $L$ equipped with a basepoint. The key observation (see also \cite{OSBr}) is that the resolution of crossings
\begin{equation*}
\Skein{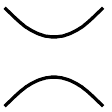}\stackrel{0}{\longleftarrow}\Skein{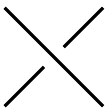}\stackrel{1}{\longrightarrow}\Skein{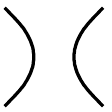}
\end{equation*}
can be realized, when thinking of the branched double cover, as a surgery triple. Here the knot along which the surgery is performed is given by the inverse image of an unknotted arc connecting the two strands in the diagram. We can then study the associated surgery spectral sequence. The main result of the present paper follows from the following.
\begin{thm}\label{main}
Suppose a diagram $D$ of a link $L$ in $S^3$ is given. Then the $E^2$-page of the link surgery spectral sequence associated to the cube of resolutions in the branched double cover is isomophic to $\BN^2(L)$.
\end{thm}

The first step is to identify the groups appearing in the $E^1$-page of the spectral sequence. These correspond to the branched double covers of the full resolutions. These are just a collection of cycles in the plane, so their branched double cover is just a connected sum $\hash^k S^2\times S^1$, where $k+1$ is the number of cycles.

\begin{lemma}\label{E1}
Let $Y=\hash^k S^2\times S^1$. We have the isomorphism of graded $\Br$-modules
\begin{equation*}
\HMItil_*(Y)=\Lambda^*(H_1(Y,\ztwo))\otimes \Br
\end{equation*}
where the $H_1(Y,\ztwo)$ lies in degree $-1$.
\end{lemma}

To perform the computations in this sections it will be helpful to reduce to the Morse-Bott setting of $\Pin$-monopole Floer homology. In particular, for a three-manifold $Y$ equipped with a self-conjugate spin$^c$ structure, it is associated in \cite{Lin} a chain complex $\check{C}(Y,\spin)$ whose homology is $\HMt_*(Y,\spin)$ such that $\jmath_*$ induces a chain automorphism. This is done by considering $\jmath$-equivariant perturbations (this necessarily introduces Morse-Bott singularities). As a Morse-Bott chain complex, it involves (generalized) singular chains in the critical submanifolds, and the differential takes into account fibered products with the compactified moduli spaces of trajectories connecting different critical submanifolds. We can use such a chain complex to define $\widetilde{CI}_*(Y,\spin)$: while we will still need to consider a chain homotopy between $\check{U}_{\hat{\q}}$ and its conjugate, the main advantage is that $\iota$ can be taken to coincide with $\jmath$. The proof that such a chain complex is quasi-isomorphic to the one defined in the previous section is essentially contained in Chapter $3$ of \cite{Lin}.

\begin{proof}
We start by discussing the case $k=1$, for which the critical set (for a certain choice of generic $\jmath$-equivariant perturbation) can be described as follows (see Chapter $4$ of \cite{Lin}). There are two critical points $\alpha_0,\alpha_1$, which are the reducible solutions corresponding to the spin connections, and correspondingly there are in the blow up two sequences of critical submanifolds $[C_0^i]$ and $[C_1^i]$ for $i\in\mathbb{Z}$, each of them being a copy of $S^2$. Furthermore, the action of $\jmath$ is given by the antipodal map of $S^2$. This description follows because under the $\jmath$-equivariance assumption, the perturbed Dirac operators involved are quaternionic linear, hence their eigenspaces are (under the genericity assumption) two-dimensional over $\mathbb{C}$. A generator of the zero dimensional singular homology of $[C_k^i]$ lies in degree $k+4i$, and the total differential of $\check{C}(Y,\spin)$ is the sum of the differential in each critical submanifold together with the map that takes a chain $\sigma$ in $[C_1^i]$ to the chain $(1+\jmath)\sigma$ in $[C_0^i]$ (under the identifications above). The generators of the homology $\HMtil_*(S^2\times S^1)$, which is identified with $H_1(Y,\ztwo)$, can be represented as follows:
\begin{itemize}
\item the one in degree zero by a point $p_0\in [C_0^0]$;
\item the one in degree one by a point $p_1\in [C_1^0]$ plus an arc $\lambda$ in $[C^0_0]$ running from $p_1$ to $\jmath_* p_1$ (under the usual identification).
\end{itemize}
Consider now the involutive chain complex $\widetilde{CI}_*(Y,\spin)$. We denote by $Q\cdot[C_k^i]$ the corresponding critical submanifold in the summand $Q\cdot\check{C}_*(Y,\spin)$. From our description of the generators, it is clear that the map
\begin{equation*}
\HMtil_*(S^2\times S^1)\stackrel{Q\cdot\iota}{\longrightarrow} Q\cdot \HMtil_*(S^2\times S^1)
\end{equation*}
is the identity, and the result immediately follows.
\par
The case of general $k$ is essentially the same, the key point being that the manifold has positive scalar curvature (see Section $36.1$ in \cite{KM}). Fix a spin structure on $Y$, together with a basis of $H^1(Y,\ztwo)$. These choices provide an identification for the torus of flat connections
\begin{equation*}
\mathbb{T}^k=(S^1)^k
\end{equation*}
and we consider a Morse function $f$ on this torus which is the sum of the height function on each factor, so that the original spin structure is the maximum. This function has exactly $2^k$ critical points. These are identified with the set of spin structures and, via the choice of basis, with $\{0,1\}^k$. Under this identification the spin structure we started with corresponds to the string of ones.
Picking then a small perturbation as in (\ref{pertS2S1}), we see that these reducible solutions corresponding to the spin structure are the only critical points of the functional. Hence for each $\alpha\in\{0,1\}^k$ we get a sequence $[C_{\alpha}^i]$ of critical submanifolds (each identified with $S^2$), whose generator of the zero dimensional singular homology has absolute grading $l(\alpha)+4i$, where $l(\alpha)$ is the number of ones in $\alpha$. Consider the lexicographical order on on $\{0,1\}^k$. We then have the following description under the identification of $\Lambda^*(H_1(Y,\ztwo))$ with $\{0,1\}^k$: the generator corresponding to $\alpha$ is represented by a single point in $[C_{\alpha}^0]$, together with other chains in critical submanifolds $[C_{\beta}^k]$ with $\beta<\alpha$. Hence again
\begin{equation*}
\HMtil_*(Y)\stackrel{Q\cdot\iota}{\longrightarrow} Q\cdot \HMtil_*(Y)
\end{equation*}
is the identity, and the result follows.
\end{proof}

From this description, we can also describe the generator in $\HMItil_*(Y)$ corresponding to $\alpha$: it consists of a point $p$ in $[C_{\alpha}^0]$, an arc $\lambda$ in $Q\cdot[C_{\alpha}^0]$ connecting $p$ and $\jmath_* p$, and other chains in $[C_{\beta}^i]$ and $Q\cdot[C_{\beta}^i]$ for $\beta<\alpha$. With this in hand, we can show the following (which is ultimately the reason why Bar Natan's homology arises in this context).

\begin{lemma}\label{S2S2}
Consider the cobordism $S^2\times S^2\setminus (D^4\amalg D^4)$ from $S^3$ itself. Then the induced map on $\HMItil_*(S^3)\equiv\Br$ is multiplication by $Q$.
\end{lemma}
\begin{proof}
This follows from the following characterization of the moduli spaces on $S^2\times S^2$ (Proposition $3$ in \cite{Lin2}): after applying a suitable $\jmath$-equivariant perturbation, the compactified moduli space $M_i$ of solutions asymptotic to $[C_0^i]$ on both ends is a one-dimensional smooth compact manifold without boundary such that both evaluation maps are generators of the one dimensional homology of the quotient (which is a copy of $\mathbb{R}P^2$). In particular, it consists of an odd number of circles. As discussed above, the generator of $1\in\Br$ has a representative given by a point in $p\in[C_0^0]$ together with an arc $\lambda$ in $Q\cdot[C_0^0]$ connecting $p$ to $\jmath_* p$. Under transversality assumptions, the moduli space $M_i$ will intersect this arc in an odd number of points, so that it defines the generator of $Q\cdot[C_0^0]$ in the outgoing end.
\end{proof}

As pointed out in the introduction, the identification provided in the previous result is not canonical, but relies on the choice of a spin structure on $Y$. By a classical observation of Turaev (\cite{Tur}), there is a natural bijection between spin structures on $\Sigma(L)$ and the \textit{quasi-orientations} of $L$, i.e. orientations of $L$ up to global reversal. Roughly speaking, the unique spin structure on $S^3$ induces by pull back a spin structure on the double cover of $S^3\setminus L$. This does not extend to the whole $\Sigma(L)$, but it does after it is suitably twisted. Furthermore, these twistings can be identified with the quasi-orientations of $L$.
\par
For each full resolution of $D$, which provides a collection of cycles in the plane, we orient each of the cycles according to the parity of the number of cycles it is cointained it. More precisely, we can color the domains in the complement of the plane in black and white so that
\begin{itemize}
\item the unbounded component is black;
\item each curve divides a black component for a white component.
\end{itemize}
We then orient the cycles as the boundaries of the white domains. This choice is particularly convenient when studying the maps induced by merging/splitting of cycles, as one readily check that there is an orientation on the corresponding handle attachment cobordism $C$ that induces the given orientations on the boundaries. Passing on the branched double covers, this implies that (as in the bijection described above) on the branched double cover $\Sigma(C)$ there is a natural spin structure inducing the given ones on the boundary.
\begin{example}
Consider the merging of two circles into one. The corresponding branched double cover is $(D^2\times S^2)\setminus D^4$. This cobordism only has one spin structure, while the boundary component $S^2\times S^1$ has two. The choice of orientations of the link and surface above implies that the spin structure on the cobordism restricts to the given one on the boundary.
\end{example}
We are now ready to describe the maps that arise as the differentials on the $E^1$-page of the link spectral sequence. For $Y=\hash^k S^2\times S^1$ we can fix a standard basis of $H_1(Y)$ given by $[\gamma_i]$, $i=1,\dots, k$, where $[\gamma_i]$ is a standard circle in the $i$th summand.
\begin{lemma}\label{maps}
Let $Y=\hash^k S^2\times S^1$. Suppose $K$ is a standard circle in one of the factors, and let $Y'=Y_0(K)$ (which can be identified with $\hash^{k-1} S^2\times S^1$). Without loss of generality, we can assume $[K]=[\gamma_1]$. Suppose we have fixed a spin structure on the corresponding cobordism $W'$, and identify the Floer groups of the boundaries as in Lemma \ref{E1} using the induced spin structure. Then the map of $\Br$-modules
\begin{equation*}
\HMItil_*(W):\HMItil_*(Y)\rightarrow \HMItil_*(Y')
\end{equation*}
acts for $\xi\in\Lambda^*(H_1(Y'))$ as
\begin{align*}
[\gamma_1]\wedge \xi&\mapsto Q\cdot \xi\\
\xi&\mapsto\xi
\end{align*}
where we naturally identify $[\gamma_i]$ for $i\geq 2$ as a basis for $H_1(Y')$.
\par
Dually, let $[K]$ be an unknot and consider $Y''=Y_0(K)$ (which can be identified with $\hash^{k+1} S^2\times S^1$). Denote by $[\gamma_0]$ the class of the new generator, and again assume that the identifications of the groups on the boundary are respect to spin structure induced by a spin structure on the cobordism $W''$. Then the map of $\Br$-modules
\begin{equation*}
\HMItil_*(W''):\HMItil_*(Y)\rightarrow \HMItil_*(Y'')
\end{equation*}
acts for $\xi\in \Lambda^*(H_1(Y))$ as
\begin{equation*}
\xi\mapsto [\gamma_0]\wedge\xi,
\end{equation*} 
where again we identify the $[\gamma_i]$ with the corresponding generators in $H_1(Y'')$.
\end{lemma}
\begin{proof}
Consider first the case in which $k=0$ and the cobordism is $W''$. Choose a metric with positive scalar curvature on $S^2\times S^1$. There are two reducible critical points $[\alpha_0]$ and $[\alpha_1]$, corresponding to the two spin connections of spin structures $\spin_0$ and $\spin_1$. Suppose that $\spin_0$ is the spin structure induced by the cobordism. We can then add equivariant regular perturbations as in \cite{Lin} so that the only critical points are the $[\alpha_i]$, and $[\alpha_0]$ lies in degree one less with respect to $[\alpha_1]$. As we are only interested in reducible moduli spaces, and as the map induced by the cobordism has degree $-1$, the choice we have made implies that the moduli spaces on the cobordism can be made transverse without recurring to $\jmath$-equivariant \textsc{asd}-perturbations. In particular, the relevant moduli spaces are all two dimensional and reducible (lying over the spin connection on the cobordism) and consist of copies of $S^2$ for which the two evaluation maps at the endpoints are diffeomorphism. The case of general $k$ is proved without significant modifications.
\par
To compute the map induced by $W'$, we can suppose that the curve $K$ is the cycle of the new $S^2\times S^1$ summand obtained from a cobordism of the form $W''$. Fix a spin structure on the ends of $W'$ and $W''$ diffeomorphic to $\hash^{k-1} S^2\times S^1$, and equip $W'$ and $W''$ with the unique spin structure extending them. We can then form the spin cobordism
\begin{equation*}
W'\cup_{\varphi} W''
\end{equation*}
where $\varphi$ is an orientation-preserving diffeomorphism between the boundary components diffeomorphic to $\hash^{k} S^2\times S^1$ identifying the induced spin structures. Notice that the result is \textit{not} the composition of the cobordisms in the surgery exact triangle. Indeed, the composition is spin, and it is diffeomorphic to the product $[0,1]\times (\hash^k S^2\times S^1)$ connected sum with $S^2\times S^2$. Combining Lemma $3$ with the usual neck stretching argument we obtain that the induced map is multiplication by $Q$. The result then follows from the computation of $W''$.
\end{proof}

\begin{proof}[Proof of Theorem \ref{main}]
We just need to put together the various pieces discussed above. Fix a pointed diagram $D$. A full resolution gives rise to a disjoint union of circles in the plane $S_0,\dots, S_k$, where we suppose that $S_0$ is the one containing the basepoint. The branched double cover is a copy of $\hash^{k}S^2\times S^1$, and the preimage of a standard arc connecting $S_0$ to $S_i$ gives rise to a generator $[\gamma_i]$ in the $i$th component. The orientation convention we have chosen fixes a spin structure on the branched double cover, hence an identification of the associated group as in Lemma \ref{E1}. This is readily identified with the corresponding (reduced) vector space in Bar-Natan's chain complex via the isomorphism of $\Br$-modules
\begin{equation*}
\varphi: \HMItil_*(\hash^{k}S^2\times S^1)\rightarrow V^{\otimes k+1}\otimes\ztwo[u]/u^2
\end{equation*}
defined as follows: for a primitive element $\xi\in \Lambda^*(H_1(\hash^{k}S^2\times S^1))$, we assign $\vp$ to the pointed circle and to the $i$th circle $\vm$ if $[\gamma_i]$ appears in $\xi$, and $\vp$ otherwise, and finally we identify $Q$ with $u$.
\par
Consider now merging/splitting cobordisms: there are in total four cases, depending on whether the circle with the basepoint is involved or not. We focus on the case in which the circles with basepoints are not involved, as the other two are analogous. Consider the case of the splitting of a circle $S_1$ into two circles $S_1$ and $S_1'$. This is obtained by doing surgery on a small unknot, and the natural choice for the generator new component, so that the description of Lemma \ref{maps} holds, is the preimage $[\gamma']$ of a standard arc connecting $S_1$ to $S_1'$ (see Figure \ref{arcs}).
\par
Indeed, our convention on orientations is compatible with the merging of the circles, hence there is a natural spin structure on the branched double cover inducing the given ones on the boundary. In this basis (which differs from the one discussed above), the map is given as in Lemma \ref{maps} by the wedge product with $[\gamma']$. We then define the the change of basis of $\HStil_*(\hash^k S^2\times S^1)$ as ($\Br$-modules), for $\xi\in \Lambda^*(\ztwo\langle [\gamma_2],\dots, [\gamma_k]\rangle)$ as
\begin{align*}
\xi&\mapsto \xi \\
[\gamma_1]\wedge \xi&\mapsto[\gamma_1]\wedge\xi\\
[\gamma']\wedge\xi&\mapsto [\gamma_1]\wedge\xi+[\gamma_1']\wedge\xi+Q\cdot \xi\\\
[\gamma_1]\wedge[\gamma']\wedge\xi&\mapsto[\gamma_1]\wedge[\gamma_1']\wedge\xi,
\end{align*}
where the old basis is $[\gamma_1],[\gamma'],[\gamma_2],\dots,[\gamma_k]$ and the new one is $[\gamma_1],[\gamma_1'],[\gamma_2],\dots,[\gamma_k]$. The asymmetric nature of this map depends on the choice we have made of considering $[\gamma_1]$ as the generator before the splitting. Under this identification, the map corresponding to the splitting is exactly the map $m$ in Bar-Natan's \textsc{tqft}: the most interesting is the one arising from the third line of the identification above, which implies the coproduct
\begin{equation*}
\vp\mapsto \vp\otimes \vm+\vm\otimes \vp+u\vp\otimes\vp.
\end{equation*}
For the merging, we are reading Figure \ref{arcs} from right to left, and the maps provided by Lemma \ref{maps} are readily identified with the multiplication $m$.
\end{proof}

\begin{figure}
  \centering
\def\svgwidth{0.8\textwidth}
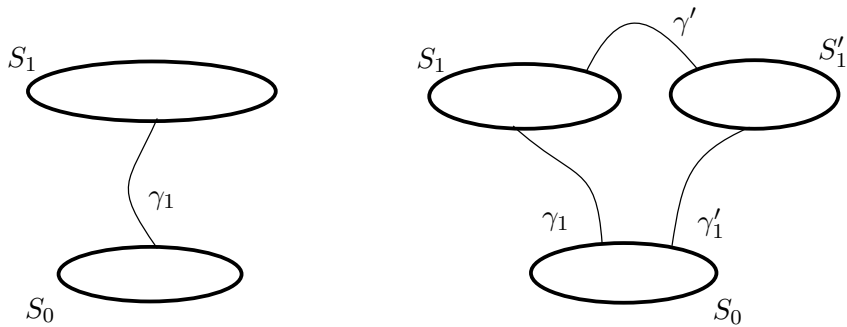
\caption{Reading the figure from left to right we have a splitting. The natural generator of the new copy $S^2\times S^1$ summand is the preimage of the arc $\gamma'$ going from $S_1$ to $S_1'$, and we need to perform a change of basis (as $\Br$-modules) to return to the usual basis induced by $\gamma_1$ and $\gamma_1'$. The circles are oriented so that the induced cobordism has a compatible orientation.}
\label{arcs}
\end{figure}

\vspace{1.5cm}
\bibliographystyle{alpha}
\bibliography{biblio}
\end{document}